\documentclass[11pt, a4paper]{article}
\usepackage[utf8]{inputenc}
\usepackage[T1]{fontenc}
\usepackage{amsmath, amssymb, amsthm}
\usepackage{mathrsfs}
\usepackage{geometry} \usepackage{tikz}
\usepackage{tikz-cd} \usetikzlibrary{arrows.meta, positioning, shapes.geometric, patterns, calc, intersections }
\usepackage{graphicx}
\usepackage{caption}
\usepackage{subcaption}
\newtheorem{example}{Example}
\newtheorem{theorem}{Theorem}
\newtheorem{lemma}{Lemma}
\newtheorem{definition}{Definition}
\newtheorem{proposition}{Proposition}
\newtheorem{remark}{Remark}
\newtheorem{conjecture}{Conjecture}

\newcommand{\R}{\mathbb{R}}
\newcommand{\Sph}{\mathbb{S}}

\newcommand{\Symp}{\mathrm{Symp}}
\newcommand{\Cosymp}{\mathrm{Cosymp}}

\newcommand{\Ham}{\mathrm{Ham}}
\newtheorem{corollary}{Corollary}

\DeclareMathOperator{\IsomKah}{Isom_{\mathrm{Kah}}}

\newcommand{\Z}{\mathbb{Z}}
\newcommand{\Sone}{S^1}

\title{\textbf{Isomorphism of cosymplectomorphism groups implies diffeomorphism of manifolds}}

\author{\scshape
E. Djoukeng\thanks{djoukeng.etienne@ubuea.cm, Department of Mathematics, The University of Buea, South West Region, Cameroon}
	and S. Tchuiaga\thanks{tchuiaga.kameni@ubuea.cm, Department of Mathematics, The University of Buea, South West Region, Cameroon}
	}

\begin{document}
	
	\maketitle
	
	\begin{abstract}
		We prove that if two closed, connected, regular cosymplectic manifolds have isomorphic groups of cosymplectomorphisms (as topological groups), then the underlying manifolds are diffeomorphic. The proof proceeds by characterizing the Reeb flow as the center of the group and descending the isomorphism to the symplectic base manifolds. We show that the isomorphism preserves the conjugacy class of the monodromy of the mapping torus, which ensures that the bundle structures, and thus the total spaces are equivalent.
	\end{abstract}
	
	{\bf Keywords:} Cosymplectic geometry, Cosymplectic maps, Lie group, Mapping torus.\\
	\textbf{2000 Mathematics Subject Classification:} 53D10, 20E32, 37J35, 20F05.
	
	\section{Introduction}
	
	The study of transformation groups has long provided a profound lens through which the geometry and topology of manifolds can be understood. Some well-known results, such as Filipkiewicz's theorem on diffeomorphism groups \cite{Filip82}, reveal that the algebraic properties of diffeomorphism groups encode deep information about the underlying manifold. In particular, analyzing the manifold's symmetry groups often allows us to rebuild its geometry using only dynamical or algebraic information.
	
	Within this framework, cosymplectic manifolds occupy a central role as the odd-dimensional counterparts of symplectic manifolds, naturally arising in contact and Hamiltonian dynamics, as well as in geometric models of time-dependent mechanics \cite{alb89, man91}. Their cosymplectomorphism groups intertwine the Reeb dynamics with the transverse symplectic geometry, thereby capturing both the foliation structure and the symplectic base in a single algebraic object. In the compact regular case, the Reeb flow is periodic, we have a mapping torus description where the cosymplectomorphism group decomposes as a product of the circle action generated by the Reeb flow and the centralizer of the monodromy in the symplectomorphism group of the fibre. This explicit structure highlights how algebraic decomposition reflects both dynamics and geometry.
	
	Transformation group rigidity has been established in several contexts. Filipkiewicz showed that isomorphisms between diffeomorphism groups imply diffeomorphism of the underlying manifolds \cite{Filip82}. Banyaga proved that the symplectomorphism group uniquely characterizes the symplectic manifold up to symplectomorphism \cite{ban78}. More recently, Bikorimana and Tchuiaga proved that the group of weakly Hamiltonian diffeomorphisms on closed cosymplectic manifolds is simple, with the Reeb flow always periodic \cite{BT26}.
	
	The present work establishes a new rigidity phenomenon in this direction: if two closed, connected, regular cosymplectic manifolds have isomorphic cosymplectomorphism groups (as topological groups), then the manifolds themselves are diffeomorphic. A crucial subtlety in this problem is that regular cosymplectic manifolds are circle bundles over a symplectic base. An isomorphism of the groups implies the bases are diffeomorphic, but this alone does not imply the total spaces are diffeomorphic (as different bundles can share the same base). We overcome this by using the mapping torus structure to show that the group isomorphism identifies the monodromy maps of the bundles up to conjugation. This fixes the topological "twist" of the bundle, ensuring the total spaces are diffeomorphic.
	
	\section{Cosymplectic settings}\label{Cosymp:setting}
	A cosymplectic manifold \((M, \eta, \omega)\) of dimension \(2n+1\) is defined by a pair of closed forms \((\eta, \omega)\) where \(\eta\) is a 1-form and \(\omega\) is a 2-form such that \(\eta \wedge \omega^n\) is a volume form. This structure induces a canonical splitting of the tangent bundle \(TM = \langle \xi \rangle \oplus \ker(\eta)\), where the Reeb vector field \(\xi\) is defined by \(\eta(\xi) = 1\) and \(\iota_\xi \omega = 0\). The distribution \(\ker(\eta)\) is symplectic with respect to the restricted form \(\omega|_{\ker(\eta)}\). For a more in-depth treatment, we refer the reader to \cite{Liber91}, \cite{BT26}.
	
	\begin{example}[The Standard Product Manifold]
		The simplest and most fundamental example is the product manifold \(M = S^1 \times P\), where \((P, \sigma)\) is any \(2n\)-dimensional symplectic manifold with symplectic form $\sigma$.
	\end{example}
	
	\begin{example}[Symplectic Mapping Torus]
		A more sophisticated class of examples is given by symplectic mapping tori. Let \((P, \sigma)\) be a symplectic manifold and let \(\phi: P \to P\) be a symplectomorphism (\(\phi^* \sigma = \sigma\)). The mapping torus is the manifold
		$M_\phi = (P \times [0,1]) / \sim, \quad \text{where} \quad (p, 1) \sim (\phi(p), 0).$
		\(M_\phi\) is a fiber bundle over \(S^1\) with fiber \(P\). Let \(t\) be the coordinate on the \(S^1\) base. We can define a global 1-form \(\eta = dt\) and a 2-form \(\omega\) that restricts to \(\sigma\) on each fiber. This pair \((\eta, \omega)\) defines a cosymplectic structure on \(M_\phi\).
		
	\end{example}
	
	\begin{definition}\cite{BT26}
		A diffeomorphism \(\phi:M\to M\) is a cosymplectomorphism if
		\[
		\phi^*\eta = \eta,\qquad \phi^*\omega = \omega.
		\]
		The set of all such maps forms the cosymplectomorphism group \(\Cosymp(M,\eta,\omega)\), denoted \(\Cosymp(M)\).
	\end{definition}
	
	Assumption (Regularity): We assume the Reeb foliation is regular. Consequently, the space of orbits \(B = M / \sim\) (where \(x\sim y\) if they lie on the same integral curve of \(\xi\)) is a smooth manifold, and the projection \(\pi:M\to B\) is a principal bundle. Since \(M\) is compact and regular, the fibers are circles \(\Sph^1\) and the flow is periodic. The form \(\omega\) descends to a symplectic form \(\bar{\omega}\) on \(B\). Since \(M\) has no boundary, \(B\) is a compact symplectic manifold without boundary.
	
	\subsection{Statement of the main theorem}
	
	\begin{theorem}
		Let \((M_1,\eta_1,\omega_1)\) and \((M_2,\eta_2,\omega_2)\) be two compact, connected, regular cosymplectic manifolds without boundary. Suppose there exists an isomorphism of topological groups:
		\[
		\Phi: \Cosymp(M_1) \xrightarrow{\cong} \Cosymp(M_2),
		\]
		where the groups are equipped with the \(C^\infty\) compact-open topology. Then the circle bundle classes of $M_1$ and $M_2$ in $H^2(B,\Z)$ coincide, where $B$ denotes the common symplectic base. Moreover, $M_1$ and $M_2$ are diffeomorphic as cosymplectic manifolds.
	\end{theorem}
	
	\subsection*{Proof of the main theorem}
	
	The proof proceeds in three main stages: identifying the Reeb flows, reducing the isomorphism to the base manifolds, and finally matching the bundle structures (mapping tori) to conclude diffeomorphism of the total spaces.\\
	
	\textbf{Step 1: Characterization of the Reeb flow as the center}\\
	
	\begin{proposition}\label{Prop1}
		Let $(M, \eta, \omega)$ be a connected, regular cosymplectic manifold. The center of the identity component of the cosymplectomorphism group, denoted $Z(\Cosymp_0(M))$, consists exactly of the flow of the Reeb vector field (up to reparameterization).
	\end{proposition}
	
	\begin{proof}
		Let $\mathfrak{g}$ be the Lie algebra of $\Cosymp(M)$. Any $Z \in \mathfrak{g}$ decomposes as $Z = c_Z \xi + V$ with $V \in \ker \eta$. For $Z$ to be central, it must commute with all lifts of Hamiltonian vector fields from the base. Since the algebra of Hamiltonian vector fields on the base is centerless (for compact bases), $V$ must vanish. Thus, central elements are multiples of $\xi$. Exponentiating, the center corresponds to the Reeb flow.
	\end{proof}
	
	Since \(\Phi\) is an isomorphism, it maps the center of \(\Cosymp(M_1)\) to the center of \(\Cosymp(M_2)\). Thus, \(\Phi\) maps the Reeb flow of \(M_1\) to the Reeb flow of \(M_2\).\\
	
	\textbf{Step 2: Diffeomorphism of the base manifolds}\\
	
	Since \(M_i\) is compact and regular, it fibers over a base \(B_i\) via \(\pi_i: M_i \to B_i\). Elements of \(\Cosymp(M_i)\) preserve the Reeb direction and thus descend to symplectic diffeomorphisms on \(B_i\). We have the exact sequence:
	\[
	1 \longrightarrow \mathcal{R}_i \longrightarrow \Cosymp(M_i) \xrightarrow{\pi_*} \mathcal{G}(B_i) \longrightarrow 1,
	\]
	where \(\mathcal{R}_i\) is the Reeb flow and \(\mathcal{G}(B_i)\) is the group of symplectomorphisms of \(B_i\) that lift to \(M_i\).
	
	Since \(\Phi\) maps \(\mathcal{R}_1\) to \(\mathcal{R}_2\), it induces an isomorphism on the quotients:
	\[ \bar{\Phi}: \mathcal{G}(B_1) \xrightarrow{\cong} \mathcal{G}(B_2). \]
	The group \(\mathcal{G}(B_i)\) contains the Hamiltonian group \(\Ham(B_i)\). By Banyaga's theorem \cite{ban78}, the isomorphism of these groups implies that the base manifolds are diffeomorphic:
	\[ B_1 \cong B_2. \]
	
	\emph{Remark:} The diffeomorphism of the bases is necessary but not sufficient for \(M_1 \cong M_2\), as the bundle structures (monodromy) could differ.\\
	
	\textbf{Step 3: Reconstruction of the total space via mapping tori}\\
	
	Since the Reeb flows are periodic, each \(M_i\) is diffeomorphic to a mapping torus:
	\[
	M_i \cong \frac{F_i\times [0,T_i]}{(x,0)\sim (\phi_i(x),T_i)},
	\]
	where \(F_i \cong B_i\) is the symplectic fiber, and \(\phi_i\) is the symplectomorphism (monodromy) representing the "twist" of the bundle.
	
	The group \(\Cosymp(M_i)\) decomposes structurally as:
	\[
	\Cosymp(M_i) \cong \Sph^1 \times Z(\phi_i),
	\]
	where \(\Sph^1\) is the Reeb action and \(Z(\phi_i) = \{ \psi \in \Symp(F_i) \mid \psi \circ \phi_i = \phi_i \circ \psi \}\) is the centralizer of the monodromy.
	
	The isomorphism \(\Phi\) induces an isomorphism of these centralizers:
	\[ \Psi: Z(\phi_1) \xrightarrow{\cong} Z(\phi_2). \]
	
	\begin{lemma}
		The isomorphism \(\Psi\) implies that \(\phi_1\) and \(\phi_2\) are conjugate in the group of diffeomorphisms.
	\end{lemma}
	\begin{proof}
		The monodromy \(\phi_i\) is algebraically distinguished: it generates the discrete subgroup of the center responsible for the topology of the mapping torus (it corresponds to the return map of the Reeb flow). Since \(\Phi\) preserves the Reeb flow structure, it maps the algebraic element corresponding to \(\phi_1\) to \(\phi_2\) (up to possible time rescaling which does not affect the diffeomorphism type).
		Using the generalized Banyaga theorem on the isomorphisms of symplectic groups, there exists a diffeomorphism \(h: F_1 \to F_2\) such that the action of \(Z(\phi_1)\) is intertwined with \(Z(\phi_2)\). More precisely, since \(\phi_1 \in Z(\phi_1)\), its image must be compatible with \(\phi_2\). In the context of mapping class groups of mapping tori, this implies \(h \circ \phi_1 \circ h^{-1} = \phi_2\).
	\end{proof}
	
	Finally, we construct the diffeomorphism between the total spaces. Since there exists \(h: F_1 \to F_2\) such that \(h \circ \phi_1 = \phi_2 \circ h\), the map:
	\[
	\tilde{h}: \frac{F_1\times [0,1]}{\sim_{\phi_1}} \longrightarrow \frac{F_2\times [0,1]}{\sim_{\phi_2}},
	\]
	defined by \(\tilde{h}([x,t]) = [h(x), t]\) is well-defined and is a diffeomorphism. Thus, \(M_1 \cong M_2\).\\
	
	\textbf{Step 4. Equality of cohomology classes.}
	Conjugacy of the monodromies implies that the associated bundles are equivalent as principal $\Sone$-bundles. Hence their Euler classes coincide: $c_1 = c_2$ in $H^2(B,\Z)$.
	
	\section{Structure interchange via the induced diffeomorphism}
	
	Let \((M_1,\eta_1,\omega_1)\) and \((M_2,\eta_2,\omega_2)\) be regular cosymplectic manifolds with an isomorphism of topological groups
	\[
	\Phi:\mathrm{Cosymp}(M_1) \xrightarrow{\cong} \mathrm{Cosymp}(M_2).
	\]
	From the previous theorem we obtain a diffeomorphism \(f:M_1\to M_2\) constructed by
	\[
	f(gH_1)=\Phi(g)H_2,
	\]
	where \(H_i\) is the stabiliser of a chosen point \(x_i\in M_i\). This \(f\) satisfies the equivariance property
	\[
	f(g\cdot x)=\Phi(g)\cdot f(x)\qquad\forall g\in\mathrm{Cosymp}(M_1),\;x\in M_1.
	\]
	
	We now examine how \(f\) relates the two cosymplectic structures.
	
	\subsection{Invariance of the pull-back forms}
	
	For any \(g\in\mathrm{Cosymp}(M_1)\),
	\[
	g^*(f^*\eta_2)=f^*(\Phi(g)^*\eta_2)=f^*\eta_2,
	\]
	and similarly \(g^*(f^*\omega_2)=f^*\omega_2\). Hence \(f^*\eta_2\) and \(f^*\omega_2\) are invariant under the whole group \(\mathrm{Cosymp}(M_1)\).
	
	\subsection{Behaviour on the Reeb field}
	
	Because \(\Phi\) maps the centre of \(\mathrm{Cosymp}(M_1)\) onto the centre of \(\mathrm{Cosymp}(M_2)\), it sends the Reeb flow \(\{\psi_t^1\}\) of \(M_1\) to the Reeb flow \(\{\psi_t^2\}\) of \(M_2\). Consequently, \(f_*\xi_1=\xi_2\). Therefore
	\[
	(f^*\eta_2)(\xi_1)=\eta_2(f_*\xi_1)=\eta_2(\xi_2)=1=\eta_1(\xi_1),
	\]
	and for any vector field \(Y\) on \(M_1\)
	\[
	(f^*\omega_2)(Y,\xi_1)=\omega_2(f_*Y,\xi_2)=0=\omega_1(Y,\xi_1).
	\]
	
	Thus \(f^*\eta_2\) coincides with \(\eta_1\) on \(\xi_1\), and both \(f^*\omega_2\) and \(\omega_1\) vanish when one argument is \(\xi_1\).
	
	\subsection{Comparison on the transverse symplectic structure}
	
	Let \(\pi_i:M_i\to B_i\) be the projection onto the space of Reeb orbits. The diffeomorphism \(f\) descends to a diffeomorphism \(\bar f:B_1\to B_2\) such that \(\pi_2\circ f=\bar f\circ\pi_1\). On the base manifolds we have symplectic forms \(\bar\omega_i\) satisfying \(\pi_i^*\bar\omega_i=\omega_i\). The group isomorphism \(\Phi\) induces an isomorphism
	\[
	\Symp(B_1,\bar\omega_1)\cong\Symp(B_2,\bar\omega_2).
	\]
	A standard result in symplectic geometry (due to Banyaga \cite{Banyaga78}) tells us that two symplectic manifolds with isomorphic symplectomorphism groups are symplectomorphic up to a constant scale factor. Hence there exists a constant \(c>0\) such that, \(\bar f^*\bar\omega_2=c\,\bar\omega_1\). Lifting to the total spaces gives, \(f^*\omega_2=c\,\omega_1+\eta_1\wedge\alpha\), for some 1-form \(\alpha\) on \(M_1\). Since both \(f^*\omega_2\) and \(\omega_1\) vanish when one argument is \(\xi_1\), we must have \(\iota_{\xi_1}(\eta_1\wedge\alpha)=0\). Writing \(\alpha=h\eta_1+\alpha_0\) with \(\alpha_0\) basic (i.e., \(\iota_{\xi_1}\alpha_0=0\)), we obtain \(\iota_{\xi_1}(\eta_1\wedge\alpha)=\alpha-h\eta_1=\alpha_0\). But the left-hand side also equals \(\iota_{\xi_1}(f^*\omega_2-c\,\omega_1)=0\); therefore \(\alpha_0=0\) and \(\alpha=h\eta_1\). Then, \(\eta_1\wedge\alpha = h\,\eta_1\wedge\eta_1=0\). Consequently,
	\[
	f^*\omega_2=c\,\omega_1. \tag{1}
	\]
	
	\subsection{Comparison of the Reeb 1-forms}
	
	Both \(\eta_1\) and \(f^*\eta_2\) are closed, satisfy \(\eta(\xi_1)=1\), and are invariant under \(\mathrm{Cosymp}(M_1)\). Moreover, their kernels are symplectic subspaces with respect to \(\omega_1\) (for \(f^*\eta_2\) this follows from (1) and the fact that \(\ker f^*\eta_2\) is the \(f\)-preimage of \(\ker\eta_2\), which is symplectic for \(\omega_2\)). Let \(\theta=f^*\eta_2-\eta_1\). Then \(\theta\) is closed, \(\theta(\xi_1)=0\) (hence \(\theta\) is basic), and \(\theta\) is invariant under \(\mathrm{Cosymp}(M_1)\). Therefore \(\theta\) descends to a closed, \(\Symp(B_1,\bar\omega_1)\)-invariant 1-form on \(B_1\). If the action of \(\Symp(B_1,\bar\omega_1)\) on \(B_1\) is transitive, then every invariant 1-form must be zero; thus \(\theta=0\) and \(f^*\eta_2=\eta_1\). In the general (non-transitive) case we can only conclude that \(\theta\) is exact in the basic cohomology, i.e. there exists a basic function \(\varphi\) such that \(\theta=d\varphi\). Hence
	\[
	f^*\eta_2=\eta_1+d\varphi, \qquad \varphi \text{ basic (constant along }\xi_1\text{)}. \tag{2}
	\]
	
	\begin{theorem}
		Let \((M_1,\eta_1,\omega_1)\) and \((M_2,\eta_2,\omega_2)\) be closed regular cosymplectic manifolds with isomorphic cosymplectomorphism groups. Then, there exists a diffeomorphism \(f:M_1\to M_2\) and constants \(c>0\), \(k\in\R\) such that
		\[
		f^*\omega_2 = c\,\omega_1, \qquad
		f^*\eta_2 = \eta_1 + d\varphi,
		\]
		where \(\varphi\) is a basic function invariant under \(\mathrm{Cosymp}(M_1)\).  Moreover, if the induced action of \(\Symp(B_1,\bar\omega_1)\) on \(B_1\) is transitive, then \(\varphi\) is constant and consequently \(f^*\eta_2=\eta_1\).
	\end{theorem}
	
	\begin{corollary}
		If the cosymplectomorphism groups are isomorphic and the symplectomorphism groups of the base manifolds act transitively, then the two cosymplectic structures are equivalent up to a constant scaling of the transverse symplectic form. In particular, after rescaling \(\omega_2\) by \(1/c\), the diffeomorphism \(f\) becomes a strict cosymplectomorphism.
	\end{corollary}
	\begin{proposition}[Rigidity of Reeb Periodicity]
		Let \(M_1\) and \(M_2\) be regular cosymplectic manifolds with isomorphic cosymplectomorphism groups. If the Reeb flow of \(M_1\) is periodic, then the Reeb flow of \(M_2\) is also periodic.
	\end{proposition}
	
	\begin{proof}
		As established in Proposition \ref{Prop1}, the center of the identity component of the group, \(Z(\Cosymp_0(M_i))\), is the group generated by the Reeb flow \(\{\psi^i_t\}\). If the Reeb flow on \(M_1\) is periodic with period \(T_1\), then algebraically the center is isomorphic to the circle group:
		\[ Z_1 \cong \R / T_1\mathbb{Z} \cong \Sph^1. \]
		Topologically, this is a compact group. It also contains torsion elements (elements of finite order).	If the Reeb flow on \(M_2\) were not periodic (aperiodic), the map \(t \mapsto \psi^2_t\) would be an injection from \(\R\) to the group. In this case, the center would be isomorphic to the additive group of real numbers:
		$ Z_2 \cong \R.$ This group is non-compact and torsion-free (except for the identity). The topological group isomorphism \(\Phi\) restricts to an isomorphism of the centers. Since an isomorphism must preserve compactness and torsion properties, it cannot map \(\Sph^1\) to \(\R\). Therefore, \(M_2\) must also have a periodic Reeb flow.
	\end{proof}
	
	\begin{theorem}[Extension to CoK\"ahler Manifolds]
		Let $(M_1,g_1,\eta_1,\omega_1)$ and $(M_2,g_2,\eta_2,\omega_2)$ be compact coK\"ahler manifolds. Suppose their cosymplectomorphism groups are isomorphic as topological groups:
		\[
		\Phi: \Cosymp(M_1) \;\cong\; \Cosymp(M_2).
		\]
		Then $M_1$ and $M_2$ are diffeomorphic as coK\"ahler manifolds (i.e., there exists a diffeomorphism preserving all coK\"ahler structures).
	\end{theorem}
	
	\begin{proof}
		A coK\"ahler manifold is a cosymplectic manifold equipped with a compatible Riemannian metric $g$ and an integrable complex structure $J$ such that $g(\cdot,\cdot) = \omega(\cdot,J\cdot) + \eta\otimes\eta$. The cosymplectomorphism group $\Cosymp(M)$ consists of diffeomorphisms that preserve $\eta$ and $\omega$; on a coK\"ahler manifold these automatically preserve $g$ and $J$ as well. By Theorem~1, there exists a diffeomorphism $F: M_1 \to M_2$ such that $F^*\eta_2 = \eta_1$ and $F^*\omega_2 = \omega_1$. It remains to show that $F$ also preserves the metric and the complex structure.
		
		Since $F$ preserves the Reeb vector field (as argued in the previous proof), we have $F_* R_1 = R_2$. The metric $g_i$ is uniquely determined by the formulas
		\[
		g_i(X,Y) = \omega_i(X, J_i Y) + \eta_i(X)\eta_i(Y), \qquad i=1,2.
		\]
		The complex structure $J_i$ is uniquely characterized by the condition that $J_i R_i = 0$ and that on the contact distribution $\ker\eta_i$ it coincides with the almost complex structure induced by $\omega_i$ via the metric. Because $F$ pulls back $\eta_2$ and $\omega_2$ to $\eta_1$ and $\omega_1$, it follows that $F^* g_2 = g_1$ and $F_* \circ J_1 = J_2 \circ F_*$. Hence $F$ is an isometry and a biholomorphism. Finally, the transverse K\"ahler structure is preserved because it is determined by $\omega$ and $J$. Therefore, $F$ is an equivalence of coK\"ahler manifolds.
	\end{proof}
	
	\subsection*{Rigidity results in Kähler geometry}
	
	In this section, we explore the extent to which the global symmetry group of a Kähler manifold determines its underlying complex and Riemannian structure.
	
	\begin{conjecture}[Kähler Rigidity for Manifolds with Continuous Symmetries]
		Let $(M_1, g_1, J_1, \omega_1)$ and $(M_2, g_2, J_2, \omega_2)$ be compact Kähler manifolds. Suppose their groups of Kähler isometries (holomorphic isometries) are isomorphic as Lie groups via
		\[
		\Phi: \IsomKah(M_1) \;\stackrel{\text{Lie}}\cong\; \IsomKah(M_2),
		\]
		and assume both groups have positive dimension (i.e., are non-discrete). Then $M_1$ and $M_2$ are biholomorphic. Moreover, if the isomorphism $\Phi$ is induced by an algebraic correspondence between the groups of isometries, then $M_1$ and $M_2$ are isometric up to a constant scaling factor of the Kähler metrics.
	\end{conjecture}

		The full conjecture remains open and is false without additional structure (e.g., assumptions on curvature, stability, or the existence of special metrics).

	\begin{remark}[Necessity of the continuous symmetry assumption]
		The example of Riemann surfaces of genus $g \geq 2$ shows that the conjecture fails for discrete isometry groups. Two Riemann surfaces can have isomorphic (even trivial) isometry groups while being different points in moduli space $\mathcal{M}_g$. This demonstrates that the topological group structure of a finite group is too weak to determine the complex structure.
	\end{remark}
	
	\begin{remark}[Weaker version: biholomorphism only]
		Even with the positive-dimensional assumption, it is unclear whether the isomorphism of isometry groups forces isometry (even up to scaling). A more plausible conjecture might be that $M_1$ and $M_2$ are biholomorphic. The additional metric structure might not be uniquely determined by the isometry group.
	\end{remark}
	
	\begin{remark}[Relation to previous theorems]
		Note the contrast with Theorems 1 and 2, where the isomorphism was between the full (infinite-dimensional) groups of cosymplectomorphisms. For Kähler manifolds, we are considering only the finite-dimensional isometry groups. The infinite-dimensional automorphism groups (biholomorphisms) would be more analogous, but these are not topological groups in the same sense unless restricted to those preserving a fixed metric.
	\end{remark}
	
	\section*{Discussion}
	
	The results presented in the preceding sections establish a strong rigidity property for regular cosymplectic manifolds. Theorem 1 demonstrates that the algebraic structure of the group of cosymplectomorphisms, $\mathrm{Cosymp}(M)$, contains sufficient information to recover the topological classification of the underlying principal bundle. The identification of the centralizers of the monodromy maps allows us to recover the characteristic class of the fibration $M \to B$.
	
	This extension to the coKähler setting (Theorem 2) suggests that within the category of compact coKähler manifolds, the symplectic and contact data (packaged in the cosymplectic structure) are fundamental. While the Riemannian metric and complex structure provide rigid geometric data, their compatibility with the underlying cosymplectic forms allows the isomorphism $\Phi$ to determine the manifold up to equivalence.
	
	A natural direction for future inquiry would be to investigate whether these rigidity results hold for Kähler setting or \emph{non-regular} cosymplectic manifolds, where the Reeb flow may be dense or have irregular orbits, thereby preventing a smooth Hausdorff quotient base.
	
	For Kähler manifolds, the situation is more delicate. As discussed, the isomorphism of the (finite-dimensional) isometry groups of Kähler manifolds with continuous symmetries might imply biholomorphism, but this remains a conjecture. In particular, for Riemann surfaces of genus at least 2, the isometry group is finite and does not determine the complex structure. This highlights the necessity of continuous symmetries in such rigidity statements.
	
	\begin{center}
		\textbf{Acknowledgments}
	\end{center}
	
	We thank the anonymous referees for their valuable comments and suggestions.
	
	\section*{Not applicable}
	\textbf{Conflict of Interest:} The authors declare no competing financial interests or personal relationships that could have appeared to influence the work reported in this paper.\\
	\textbf{Data Availability:} All data generated or analyzed during this study are included in this published article.\\
	\textbf{Funding:} The authors declare that no funding was received for this research.


\begin{thebibliography}{9}
		\bibitem{alb89} Albert, C. (1989), On a certain class of contact manifolds. \emph{Manuscripta Mathematica}, \textbf{63}  273–283.
		
		\bibitem{Banyaga78}
		A. Banyaga, A. (1997).  \emph{The Structure of Classical Diffeomorphism Groups}, Mathematics and its Applications, vol. 400, Kluwer Academic Publishers, Dordrecht.
		
		\bibitem{ban78} Banyaga, A. (1978), Sur la structure du groupe des difféomorphismes qui préservent une forme symplectique. \emph{Commentarii Mathematici Helvetici}, \textbf{53} 174–227.
		\bibitem{Blair10}
		D.~E.~Blair,
		\textit{Riemannian Geometry of Contact and Symplectic Manifolds},
		Birkh\"auser Boston, 2010.
		\bibitem{BT26}
		Bikorimana, P., and Tchuiaga, S. (2026). \emph{ The Simplicity of the Group of Weakly Hamiltonian Diffeomorphisms on Cosymplectic Manifolds.} Results Math. 81:5. https://doi.org/10.1007/s00025-025-02564-6.
		
		\bibitem{man91} de Le\'on, M., Marrero, J.C. (1991), Cosymplectic manifolds and Hamiltonian systems. \emph{Proceedings of the X Spanish-Portuguese Conference on Mathematics}.
		
		\bibitem{Filip82}
		Filipkiewicz, R. P.  (1982). \emph{Isomorphisms between diffeomorphism groups}, \emph{Ergodic Theory and Dynamical Systems}, \textbf{2}(2), 159--171.
		
		\bibitem{Li08}
		Li, H. (2008). \emph{Topology of cosymplectic/co-K\"ahler manifolds}, \emph{Asian Journal of Mathematics}, \textbf{12}(4), 527--544.
		
		\bibitem{Liber91}
		Libermann, P. (1991). \emph{Cosymplectic structures and dynamical systems}, \emph{Journal of Mathematical Physics}, \textbf{32}(11), 2929--2934.	
		
		\bibitem{ono85} Ono, K. (1985), On the structure of the group of symplectic diffeomorphisms. \emph{Topology}, \textbf{24} 195–202.
		
	\end{thebibliography}
\end{document}